\documentclass[11pt,a4paper,reqno]{amsart}
\usepackage{amsmath,amscd,amssymb,latexsym}

\newtheorem{Thm}{Theorem}[section]
\newtheorem{Cor}[Thm]{Corollary}
\newtheorem{Lem}[Thm]{Lemma}

\theoremstyle{definition}

\theoremstyle{remark}

\def\C{{\mathbb C}}
\def\F{{\mathbb F}}

\def\P{{\mathbb P}}
\def\Q{{\mathbb Q}}

\begin{document}
\title[Modular curves with infinitely many cubic points]{Modular curves with infinitely many\\ cubic points}
\author[D. Jeon]{Daeyeol Jeon}
\address{Daeyeol Jeon, Department of
Mathematics education, Kongju National University, 56
Gongjudaehak-ro, Gongju-si, Chungcheongnam-do 314-701, South Korea}
\email{dyjeon@kongju.ac.kr}
\thanks{This research was supported by Basic Science Research Program through the National Research Foundation of Korea (NRF) funded by the Ministry of Education (2014R1A1A2056390).}

\begin{abstract}
In this study, we determine all modular curves $X_0(N)$ that admit infinitely many cubic points.
\end{abstract}

\maketitle \setcounter{section}{-1}
\section{Introduction}


A curve $X$ defined over an algebraically closed field $k$ is called {\it $d$-gonal} if it admits a map $\phi:X\to \P^1$ over $k$ of degree $d.$ 
The smallest possible value of $d$ is called the {\it gonality} of $X$, and is denoted by ${\rm Gon}(X)$.
If a curve $X$ is two-gonal and of genus $g(X)> 1$, then $X$ is said to be {\it hyperelliptic}. If a curve $X$ is three-gonal, then we call $X$ {\it trigonal}. 

Frey \cite{Fr} proved that if a curve $X$ over a number field $K$ has infinitely many points $P$ such that $[K(P):K]\leq d$, then ${\rm Gon}(X)\leq 2d$.

On the other hand, Abramovich and Harris \cite{A-H} conjectured that a curve $X$ defined over $K$ has infinitely many points $P$ satisfying $[K(P):K]\leq d$ if and only if it admits a $K$-rational map of degree not more than $d$ onto $\P^1$ or an elliptic curve with positive $K$-rank. This was conjectured when the authors were considering the questions raised by Harris and Silverman in \cite{H-S}. 
Indeed, they proved that this conjecture holds for $d=2$ or $3$, and for $d=4$ when $g(X)\neq 7$. 
However, this conjecture is false in general, as shown in \cite{D-F}.

For a positive integer $N$, consider the congruence subgroup $$\Gamma_0(N):=\left\{\begin{pmatrix}
a&b\\c&d\end{pmatrix} \in\operatorname{SL}_2(\mathbb Z)\,|\,c \equiv0\mod N \right\}.$$ 
Let $X_0(N)$ denote the modular curve corresponding to $\Gamma_0(N)$, and let $g_0(N)$ denote its genus. 
The modular curve $X_0(N)$(with cusps removed) parametrizes isomorphism classes of elliptic curves with cyclic $N$-isogenies.


Ogg \cite{O} determined all hyperelliptic curves $X_0(N)$.
Bars \cite{B} determined all bielliptic curves $X_0(N)$, and by combining his results with Ogg's, he found all of the curves $X_0(N)$ that contain infinitely many points defined over quadratic number fields.
Here, a {\it bielliptic curve} means a curve of genus greater than one that admits a map of degree two onto an elliptic curve.





Hasegawa and Shimura \cite{Ha-Sh} determined all trigonal curves $X_0(N)$.
In this paper, by using their results, we shall determine all curves $X_0(N)$ that have infinitely many cubic points. 
A {\it cubic point} is a point that can be defined over a cubic number field.
Our main result is as follows.

\begin{Thm}\label{main}
A curve $X_0(N)$ has infinitely many cubic points if and only if
$$N\in\{1-29,31,32,34,36,37,43,45,49,50,54,64,81\}.$$
\end{Thm}

The following result follows directly from our main theorem.

\begin{Cor} If $K$ varies over all cubic number fields and $E$ varies over all elliptic curves over $K$, there exist infinitely many absolutely non-isomorphic elliptic curves $E$ defined over $K$ with $K$-rational cyclic $N$-isogenies if and only if $N$ is one of the numbers stated in Theorem \ref{main}. Here absolutely non-isomorphic means non-isomorphic over the algebraic closure, in other words, it means having different $j$-invariants.
\end{Cor}

\section{Preliminaries}

Let $X$ be a smooth projective curve defined over a number field $K$.
For any non-negative integer $d$, let ${\rm Pic}^d(X)$ be the scheme parametrizing isomorphism classes of line bundles of degree $d$ on $X$, and let $J(X)$ be the Jacobian of $X$, which is equal to ${\rm Pic}^0(X)$.
For any point $x\in {\rm Pic}^d(X)$, we write $L_x$ for a line bundle of degree $d$ on $X$ associated to $x$.
For any non-negative integer $r$, we set $W_d^r(X)=\{x\in{\rm Pic}^d(X)\,|\,h^0(X,L_x)>r\}$, endowed with its usual scheme structure, and we write $W_d(X):=W_d^0(X)$. 

Suppose that $X$ admits a $K$-rational point.
We note that $X_0(N)$ always has a $\Q$-rational point.
We say that a point $P$ of $X$ has {\it degree $d$ over $K$} if $[K(P):K]=d$.
Let $X^{(d)}$ be the $d$-th symmetric product of $X$.
Note that any point of $X$ with degree $\leq d$ over $K$ gives rise to a $K$-rational point of $X^{(d)}$ (cf. \cite{Fr}).
Thus the set of $K$-rational points $X^{(d)}(K)$ is infinite if and only if $X$ has infinitely many points of degree $\leq d$ over K.

Suppose that $X$ admits no maps of degree at most $d$ to $\P^1$ over $K$.
According to Proposition 1 of \cite{Fr}, $X^{(d)}(K)$ can be embedded in ${\rm Pic}^d(X)$ as $W_d(X)(K)$.
If $W_d(X)$ contains no translates of abelian subvarieties of ${\rm Pic}^d(X)$, then $X^{(d)}(K)$ is finite as a result of Faltings' Theorem in \cite{Fa1}.

Conversely, if $X$ admits a $K$-rational map of degree $d$ onto $\P^1$, then $X^{(d)}(K)$ is infinite.

For the dimension of abelian varieties contained in $W_d(X)$, we refer the following theorem.

\begin{Thm}\cite{D-F}\label{DF}
Let $X$ be a smooth curve of genus $g$ such that $W_d^r(X)$ contains an abelian variety $A$, and assume that $d\leq g-1+r$. Then, $\dim(A)\leq \frac{d}{2}-r$.
\end{Thm}

Now, we consider the case with $d=3$. 
By the above theorem, the only non-zero abelian varieties that can be contained in $W_3(X)$ are elliptic curves.

When dealing with an individual curve, the following facts are
useful.

\begin{Thm}\label{castelnuovo}{\rm{\textbf{(Castelnuovo's Inequality)}}}
Let $F$ be a function field with perfect constant field $k$.
Suppose that there are two function subfields $F_1$ and $F_2$ with constant field $k$ satisfying
\begin{enumerate}
\item[(1)] $F=F_1F_2$ is the compositum of $F_1$ and $F_2.$
\item[(2)] $[F:F_i]=n_i,$ and $F_i$ has genus $g_i$ $(i=1,2).$
\end{enumerate}
Then, the genus $g$ of $F$ is bounded by
$$g\leq n_1g_1+n_2g_2+(n_1-1)(n_2-1).$$
\end{Thm}

We make use of Ogg's result \cite{O} concerning a lower bound on the number of $\F_{p^2}$-rational points on the reduction $X_0(N)_{p}$ of $X_0(N)$ modulo $p$, where $p$ is a prime with $p\nmid N$.
Hasegawa and Shimura \cite{Ha-Sh} formulate Ogg's result as follows:

\begin{Lem}\label{Ogg}\cite[Lemma 3.1]{Ha-Sh} For a prime $p$ with $p\nmid N$, let $|X_0(N)_{p}(\F_{p^2})|$ denote the number of $\F_{p^2}$-rational points on $X_0(N)_{p}$.
Then, $$|X_0(N)_{p}(\F_{p^2})|\geq L_p(N):=\frac{p-1}{12}\varphi(N)+2^{\omega(N)},$$
where $\psi(N)=N\displaystyle\prod_{{r|N}\atop{r\,{\rm prime}}}\left(1+\frac{1}{r}\right)$ and $\omega(N)$ is the number of prime factors of $N$.
\end{Lem}


For the explicit computation of $|X_0(N)_{p}(\F_{p^2})|$, let us employ the zeta function of $X_0(N)$, for which we refer to \cite{Mi}.
Let $X$ be a complete nonsingular curve of genus $g$ over $\Q$.
For all but finitely many primes $p$, the reduction $X_p$ of $X$ modulo $p$ will be a complete nonsingular curve over $\F_p$.
We call the primes for which this is true the {\it good primes} for $X$ and we call the remainder the {\it bad primes}.
Suppose that $p$ is a good prime.
Furthermore, define $Z(X_p,t)$ to be the power series with rational coefficients such that
\begin{equation}\label{zeta1}
\log Z(X_p,t)=\sum_{n=1}^\infty |X_p(\F_{p^n})|\frac{t^n}{n}.
\end{equation}
Define $\zeta(X_p,s)=Z(X_p,p^{-s})$ and set
$$\zeta(X,s)=\prod_p\zeta_p(X,s),$$
where $\zeta_p(X,s)=\zeta(X_p,s)$ when $p$ is a good prime and it is as defined as in \cite{Se} when $p$ is a bad prime.
Then, we can write
\begin{equation}\label{zeta2}
\zeta(X,s)=\frac{\zeta(s)\zeta(s-1)}{L(X,s)},
\end{equation}
where $\zeta(s)$ is the Riemann zeta function and
$L(X,s)$ is the Hasse-Weil $L$-function of $X$.

On the other hand, consider a cusp form
\begin{equation}\label{cuspf}
f(\tau)=\sum_{n=1}^\infty a_n q^n\,\,(q=e^{2\pi i\tau})
\end{equation}
of weight 2 for $\Gamma_0(N)$, which is a normalized eigenform for all of the Hecke operators $T_p$ for primes $p$ with $p\nmid N$.
In addition, one can define the $L$-function
\begin{equation}\label{Lf}
L(f,s)=\sum_{n=1}^\infty\frac{a_n}{n^s}.
\end{equation}
The two series expressions of \eqref{cuspf},\eqref{Lf} are related by the Mellin transform and its inversion formula.
The $L$-function admits the following product formula:
\begin{equation}\label{l-function1}
L(f,s)=\prod_p(1-a_pp^{-s}+p^{1-2s})^{-1}.
\end{equation}

The Eichler-Shimura theorem implies the following result.

\begin{Thm}\cite{Mi}\label{l-function2} Let $\{f_1,\dots,f_g\}$ be a basis for the cusp forms of weight 2 for $\Gamma_0(N)$, chosen to be normalized eigenforms for all of the Hecke operators $T_p$ for primes $p$ with $p\nmid N$.
Then, apart from the factors corresponding to a finite number of primes, $L(X_0(N),s)$ is equal to the product of $L(f_i,s).$
\end{Thm}

For the reader's convenience, we provide lists of the integers $N$ for which $X_0(N)$ is rational, elliptic, hyperelliptic, bielliptic, or of gonality 3.

\begin{Thm}\label{list}\cite{B,Ha-Sh,O} The following holds:
\begin{enumerate}
\item[(a)] $X_0(N)$ is rational if and only if $N\in\{1-10, 12, 13, 16, 18, 25\}$.
\item[(b)] $X_0(N)$ is elliptic if and only if $N\in\{11, 14, 15, 17, 19, 20, 21, 24, 27$, $32, 36, 49\}$.
\item[(c)] $X_0(N)$ is hyperelliptic if and only if $N\in\{22, 23, 26, 28, 29, 30, 31, 33$, $35, 37, 39, 40, 41, 46, 47, 48, 50, 59, 71\}$.
\item[(d)] $X_0(N)$ is bielliptic if and only if $N\in\{22, 26, 28, 30, 33, 34, 35, 37, 38$, $39, 40, 42, 43, 44, 45, 48, 50, 51, 53, 54, 55, 56, 60, 61, 62, 63, 64, 65, 69, 72$, $75, 79, 81, 83, 89, 92, 94, 95, 101, 119, 131\}$.
\item[(e)] ${\rm Gon}(X_0(N))=3$ if and only if $N\in\{34, 38, 43, 44, 45, 53, 54, 61, 64, 81\}$.
\end{enumerate}
\end{Thm}

\section{Trielliptic curves over $\Q$}

If a curve $X$ of genus $g(X)>1$ admits a map to an elliptic curve of degree 3, we call $X$ {\it trielliptic}.
If such a degree 3 map can be defined over $\Q$, we call $X$ trielliptic over $\Q$.

Suppose that $X_0(N)$ has infinitely many cubic points, where $N$ is not contained in the lists (a)-(e) of Theorem \ref{list}.
Then, by Theorem 1 in \cite{A-H}, $X_0(N)$ admits a $\Q$-rational map to $\P^1$ of degree at most 3 or to an elliptic curve with positive $\Q$-rank.
From now on, the rank of an elliptic curve always mean the $\Q$-rank.
In fact, this map must be a map to an elliptic curve of degree 3, because the other possibilities are the curves listed in Theorem \ref{list}.
Thus, if $X_0(N)$ is not trielliptic and is not listed in Theorem \ref{list}, it can have only finitely many cubic points.

First, we will obtain a lower bound on the value of $N$ for which $X_0(N)$ is not trielliptic over $\Q$.

Let $E$ be an elliptic curve over a finite field $\F_{p^k}$.
The Hasse bound is
\begin{equation}\label{Hasse}
\big | |E(\F_{p^k})|-(p^k+1)\big |\leq 2\sqrt{p^k}.
\end{equation}

Suppose there exists a $\Q$-rational map $f:X_0(N)\to E$ of degree 3 and a prime $p\nmid N$.
Because $p\nmid N$, the curve $X_0(N)$ has good reduction at $p$.
Note that the conductor ${\rm Cond}(E)$ of $E$ divides $N$, hence $E$ also has good reduction at $p$.
Thus, $f$ induces a $\F_p$-rational map
$$\tilde f:X_0(N)_{p}\to E_p$$
of degree $3$, where $E_p$ is the reduction of $E$ at $p$ (\cite{N-Sa}).
By applying the Hasse bound, we obtain
$$|E_p(\F_{p^2})|\leq (p+1)^2,$$
and hence
\begin{equation*}
|X_0(N)_{p}(\F_{p^2})|\leq U_p(N):=3(p+1)^2.
\end{equation*}
By this bound and Lemma \ref{Ogg}, the following inequality for trielliptic curves $X_0(N)$ over $\Q$ holds for all $p\nmid N$:
\begin{equation}\label{ineq}
\frac{p-1}{12}\varphi(N)+2^{\omega(N)}\leq |X_0(N)_{p}(\F_{p^2})|\leq 3(p+1)^2.
\end{equation}

By a proof that is exactly the same as for Lemma 3.2 of \cite{Ha-Sh}, we have the following result.

\begin{Lem}\label{300} If $N\geq 300$, the curve $X_0(N)$ cannot be trielliptic over $\Q$.
\end{Lem}

Now, consider the integers $N<300$ that are not contained in any of the lists in Theorem \ref{list}.
Then, ${\rm Gon}(X_0(N))>3$.
We will prove that $X_0(N)$ has only finitely many cubic points.

For this purpose, we exclude many values of $N$ for which $X_0(N)$ cannot have infinitely many cubic points by using the methods described below.
First, we can exclude some $N$ by using \eqref{ineq}.

By Theorem \ref{DF}, $X_0(N)$ has infinitely many cubic points only if $W_3(X_0(N))$ contains a translate of an elliptic curve with positive rank.
Because such an elliptic curve has conductor dividing $N$, it should appear in the table of Cremona \cite{C}.
Thus, one can exclude the integers $N$ for which there exist no elliptic curve with positive rank whose conductor divides $N$.

Let $E$ be a strong Weil curve, and let $\varphi:X_0(N)\to E$ be its strong Weil parametrization.
If $E'$ is isogenous to $E$ over $\Q$ and a map $f:X_0(N)\to E'$ is given, then there exists an isogeny $g:E\to E'$ such that $f=g\circ\varphi$, and hence $\deg(\varphi)$ divides $\deg(f)$.
Thus, if $\deg(\varphi)>3$ for the strong Weil parametrizations $\varphi$ of all strong Weil curves with positive rank, then $X_0(N)$ does not admit a map of degree 3 to an elliptic curve with positive rank and conductor $N$.
One can find such strong Weil parametrizations and their degrees in Table 22 of \cite{C}.
Suppose that there is no elliptic curve of positive rank with conductor $M|N$ with $M<N$. Then, $X_0(N)$ cannot have infinitely many cubic points.

By using the three methods described above, one can exclude all values of $N$ except for the following:
\begin{equation}\label{numbers}
\begin{array}{ll}
N\in&\{74,86,106,111,114,116,122,129,158,159,\\
&164,166,171,172,185,215\}
\end{array}
\end{equation}

Suppose that $X_0(215)$ admits a $\Q$-rational map of degree 3 from $X_0(215)$ to an elliptic curve $E$.
Then ${\rm Cond}(E)$ divides $215$.
According to the table of Cremona \cite{C}, $E$ should be equal to one of 43A1 and 215A1.
Because the degree of the strong Weil parametrization of 215A1 is 8, there does not exist a map of degree 3 from $X_0(215)$ to 215A1 over $\Q$.
Suppose that there is a map of degree 3 from $X_0(215)$ to 43A1 over $\Q$.
One can easily compute that the number of $\F_{4}$-rational points of the reduction of $43A1$ at 2 is 5.
Thus, $|X_0(215)_2(\F_4)|\leq 15$, but $L_2(215)=26$ which is a contradiction.
Thus $X_0(215)$ is not trielliptic over $\Q$.
By the exact same method, one can exclude all of the numbers in \eqref{numbers} except for $N\in\{86,122,158,159\}$.

Consider the curve $X_0(159)$ of genus 17. 
Suppose that there exists a map of degree 3 from $X_0(159)$ to an elliptic curve $E$.
Note that the quotient space $X_0(159)/W_{159}$ of $X_0(159)$ by the full Atkin-Lehner involution $W_{159}$ is of genus 4.
Let $F$(resp. $F_1$, $F_2$) be the function field of $X_0(159)$(resp. $X_0(159)/W_{159}$, $E$). 
Then by applying Castelnuovo's inequality (Theorem \ref{castelnuovo}), one obtains a contradiction.
Thus, $X_0(159)$ is not trielliptic.
Using the same method, one can prove that $X_0(158)$ is not trielliptic.
For the curve $X_0(158)$ of genus 19, one can consider the quotient space $X_0(158)/W_{79}$ of $X_0(158)$ by the partial Atkin-Lehner involution $W_{79}$, whose genus is 5.

Finally we consider $X_0(86)$ and $X_0(122)$.
To deal with these curves, we explicitly compute $|X_0(86)_3(\F_{9})|$ and $|X_0(122)_3(\F_{9})|$.
First, let us consider $X_0(86)$ of genus $10$.
Suppose that $X_0(86)$ has infinitely many cubic points.
Then, there exists a map $X_0(86)\to E$ of degree 3 where $E$ is an elliptic curve of ${\rm rank}(E)>0$. 
Because there are no elliptic curves over $\Q$ of conductor $86$,
the elliptic curve $E$ is isomorphic to 43A1.
We compute that the number of $\F_{9}$-rational points of the reduction of $43A1$ at 3 is 12.
Thus, $|X_0(86)_3(\F_9)|\leq 36$.

From the Stein's modular form database \cite{St}, we can compute a basis $\{f_1,\dots,f_{10}\}$ for the cusp forms of weight 2 for $\Gamma_0(86)$, which consists of normalized eigenforms for the Hecke operators.
By Theorem \ref{l-function2}, \eqref{zeta2}, and \eqref{l-function1}, we obtain the equality
$$Z(X_0(86)_3,t)=\frac{1}{(1-t)(1-3t)}\prod_{i=1}^{10}(1-\alpha_it+3t^2),$$
where the $\alpha_i$ are eigenvalues of $T_3$.
Note that the $\alpha_i$ are the coefficients of $q^3$ in the $q$-expansions of the $f_i$.
Thus, we have that
\begin{align*}
Z(X_0(86)_3,t)=&(1+2t+3t^2)^2(1+4t^2+9t^4)^2(1-t+5t^2-3t^3+9t^4)\\
&(1+t+t^2+3t^3+9t^4)/(1-t)(1-3t).
\end{align*}
From \eqref{zeta1}, we obtain that
$$|X_0(86)_3(\F_9)|=\frac{d^2}{dt^2}Z(X_0(86)_3,t)|_{t=0}=40,$$
which is a contradiction.

Suppose that $X_0(122)$ has infinitely many cubic points.
Then, there exists a map of degree 3 from $X_0(122)$ to 61A1.
Note that the number of $\F_{9}$-rational points of the reduction of 61A1 at 3 is 12, hence $|X_0(122)_3(\F_9)|\leq 36$.
However, $|X_0(122)_3(\F_9)|=40,$ which is a contradiction.

By combing all of the results described above with Lemma \ref{300} and using Abramovich and Harris' result, we obtain the following result.

\begin{Lem}\label{lem1} If $N$ is not contained in any of the lists of numbers in Theorem \ref{list}, then $X_0(N)$ has only finitely many cubic points.
\end{Lem}

\section{Infinitely many cubic points}

In this section, we always assume that $N$ is one of the numbers given in Theorem \ref{list}.
First, we determine all of the trigonal curves $X_0(N)$ that admit maps of degree 3 to $\P^1$ over $\Q$. 
We call such curves {\it trigonal over} $\Q$.
Indeed, Hasegawa and Shimura \cite{Ha-Sh} proved that $X_0(N)$ is trigonal if and only if it is of $g_0(N)\leq 2$ or a non-hyperelliptic curve of $g_0(N)=3,4$.
Moreover, they determined the minimal degree of a number field over which there is a trigonal map $X_0(N)\to\P^1$ for $g_0(N)=3,4$.
All of the curves $X_0(N)$ with ${\rm Gon}(X_0(N))=3$ and $g_0(N)=3$, where $N\in\{34,43,45,64\}$, are trigonal over $\Q$.
On the other hand, among the curves $X_0(N)$ with ${\rm Gon}(X_0(N))=3$ and $g_0(N)=4$, only $X_0(54)$ and $X_0(81)$ are trigonal over $\Q$.

If $X_0(N)$ is rational, then it is obviously trigonal over $\Q$.
If $X_0(N)$ is elliptic, then one can obtain a trigonal map defined over $\Q$ by mapping to the $y$-coordinate in a Weierstrass equation.

Let us consider the (hyperelliptic) curves $X_0(N)$ of genus 2, i.e., $N\in\{22,23,26,28,29,31,37,50\}$.
According to Lemma 2.1 of \cite{JKS}, any curve of genus $2$ that has at least three $\Q$-rational points is trigonal over $\Q$.
If $N>1$ is not a power of a prime, then $X_0(N)$ has at least three $\Q$-rational cusps by Proposition 2 in \cite{O2}. 
Thus, $X_0(N)$ is trigonal over $\Q$ for $N=22,26,28,50$.
Note that for prime $N$, $X_0(N)$ has only two $\Q$-rational cusps, i.e., the cusps $0$ and $\infty$.
However, it is well-known that $X_0(37)$ has two non-cuspidal $\Q$-rational points (see \cite[Proposition 2]{MS}.), and hence $X_0(37)$ has four $\Q$-rational points. 
Thus, $X_0(37)$ is trigonal over $\Q$.

To deal with the remaining curves $X_0(N)$ of genus 2, we recall that a {\it hyperelliptic involution} $\nu$ of a hyperelliptic curve $X$ is an automorphism of $X$ of order 2 such that the quotient space $X/\langle\nu\rangle$ is rational.
Furthermore, $\nu$ is unique, hence it is contained in the center of the automorphism group of $X$ and defined over $\Q$.

\begin{Lem}\label{genus2} Let $X$ be a curve of genus 2 over a perfect field $k$, and let $\nu$ be a hyperelliptic involution of $X$.
If $X$ has two $k$-rational points $P_1$ and $P_2$ such that $\nu(P_1)=P_2$, then there exists a map $X\to \P^1$ of degree 3 that is defined over $k$.
\end{Lem}
\begin{proof} We employ the Riemann-Roch theorem over $k$.
For a divisor $D$ on $X$, let $\ell(D)$ denote the dimension of the Riemann-Roch space $\mathcal L(D)$.
Now consider the divisor $D=P_1+P_2$.
Because $\nu(P_1)=P_2$, there exists a function on the rational quotient $X/\langle\nu\rangle$ that induces a function $f$ on $X$ having simple poles only at $P_1$ and $P_2$, and no poles elsewhere. 
Hence $f\in \mathcal L(D)$.
Because $\mathcal L(D)$ contains a constant function, $\ell(D)=2$.
Thus, $D$ is a canonical divisor.
However, $2P_1$ cannot be a canonical divisor, because then $P_1-P_2$ would be a principal divisor, which is impossible on a curve of positive genus.
Thus, a non-constant function in $\mathcal L(3P_1)$ has a triple pole at $P_1$, which defines a trigonal map on $X$ over $k$.
\end{proof}

By Theorem 2 of \cite{O}, the hyperelliptic involutions of $X_0(N)$ for $N\in\{23,29,31\}$ are the full Atkin-Lehner involutions $W_N$, which map $0$ to $\infty$. 
By Lemma \ref{genus2}, the curves $X_0(N)$ for $N=23,29,31$ are trigonal over $\Q$.

Summarizing the above, all of the curves of $g_0(N)\leq 2$ are trigonal over $\Q$, and among the curves $X_0(N)$ with ${\rm Gon}(X_0(N))=3$, $X_0(N)$ is trigonal over $\Q$ only for $N\in\{34,43,45,54,64,81\}$.
Thus we obtain the following result.

\begin{Thm}\label{tri} The curves $X_0(N)$ are trigonal over $\Q$ if and only if $N\in\{1-29,31,32,34,36,37,43,45,49,50,54,64,81\}$.
\end{Thm}

Suppose that $X_0(N)$ is not hyperelliptic and that $N$ is not in the list given in Theorem \ref{tri}.
Then, it admits no map to $\P^1$ of degree at most 3 over $\Q$, and hence it can only contain infinitely many cubic points when $W_3(X_0(N))$ contains a translate of an elliptic curve of positive rank.
From this fact, one can obtain the following result.

\begin{Lem}\label{lem2} The curve $X_0(N)$ has only finitely many cubic points for $N\in\{38,42,44,51,55,56,60,62,63,69,72,75,94,95,119\}$.
\end{Lem}
\begin{proof}
For these values of $N$, there are no elliptic curves with positive rank and conductor $M|N$.
\end{proof}

Consider the curve $X_0(92)$, which is bielliptic. 
There exists only one isogeny class that contains an elliptic curve with positive rank and conductor $92$, namely 92B1. 
There are no elliptic curves of positive rank with conductor $23$ or $46$. 
Because the degree of the strong Weil parametrization of 92B1 is 6, $X_0(92)$ admits no map of degree at most 3 to an elliptic curve of positive rank.
Thus, we have the following result:

\begin{Lem}\label{lem3} The curve $X_0(92)$ has only finitely many cubic points.
\end{Lem}

For the bielliptic curves $X_0(N)$ that admit a map of degree 2 to an elliptic curve with positive rank, we prove the following result.
\begin{Lem}\label{lem4} The curve $X_0(N)$ has only finitely many cubic points for $N\in\{53,61,65,79,83,89,101,131\}$.
\end{Lem}
\begin{proof} For such $N$, the curve $X_0(N)$ is a bielliptic curve that admits a strong Weil parametrization of degree 2 to an elliptic curve with positive rank and conductor $N$.
Define a map
$$\phi:X_0^{(3)}(N)\to J_0(N)$$
by $\phi(P_1,P_2,P_3)=[P_1+P_2+P_3-3\infty]$, where $J_0(N)$ is the Jacobian variety of $X_0(N)$, $\infty$ is an infinite cusp, and $[\,\,]$ denotes the divisor class.
Then, the image of $X_0^{(3)}(N)$ under $\phi$ can be identified with $W_3(X_0(N))$.
Because any $\Q$-rational map $X_0(N)\to E$ with ${\rm Cond}(E)=N$ factors through a strong Weil parametrization of degeee 2, $X_0(N)$ admits no map of degree 3 to an elliptic curve $E$ with ${\rm rank}(E)>0$ and ${\rm Con}(E)=N$.
Note that there is no elliptic curve $E$ with ${\rm rank}(E)>0$ such that ${\rm Con}(E)$ is a proper divisor of $N$.
Thus $X_0(N)$ admits no map of degree 3 to an elliptic curve with positive rank.
Now, suppose that $X_0(N)$ has infinitely many cubic points.
Then $W_3(X_0(N))(\Q)$ is infinite, and by \cite{Fa2} there are finitely many elements $x_1,\dots,x_n$ of $W_3(X_0(N))(\Q)$ such that
$$W_3(X_0(N))=\bigcup_{i=1}^n\left[x_i+E_i(\Q)\right],$$
where the $E_i$ are elliptic curves in $J_0(N)$.
Note that there exist infinitely many points in $W_3(X_0(N))(\Q)$ that cannot be obtained from a shift of $W_2(X_0(N))(\Q)$ by a rational point of $X_0(N)$.
Then, we can have an embedding of an elliptic curve in $W_3(X_0(N))$, say $E_1$, which maps to $x_1+E_1$ such that $x_1+E_1$ is not contained in the shift of $W_2(X_0(N))$ by any rational point of $X_0(N)$.
Thus, $x_1+E_1$ cannot be contained in the shift of $W_2(X_0(N))$ by any point of $X_0(N)$.
Suppose that there exists a point $P$ of $X_0(N)$ such that $x_1+E_1\subseteq P+W_2(X_0(N))$.
Viewing the elements of $W_3(X_0(N))$ as elements of $X_0^{(3)}(N)$, there exists an element $x=(P_1,P_2,P_3)\in x_1+E_1(\Q)$ such that $x=(P,Q_1,Q_2)$ for some $(Q_1,Q_2)\in W_2(X_0(N))$.
Because the $P_i$ are points of $X$ of degree at most $3$, so is $P$.
Because $P$ is not a rational point, $P$ is either of degree $2$ or $3$.
If $P$ is of degree $2$, then the rational points in $P+W_2(X_0(N))$ can be regarded as the elements of the shift of $W_2(X_0(N))$ by a rational point of $X_0(N)$, which is a contradiction.
If $P$ is of degree $3$, then $P+W_2(X_0(N))$ only contains finitely many rational points, because $P$ is fixed, which is again a contradiction.
Therefore, the embedding of $E_1$ into $W_3(X_0(N)$ satisfies the minimality condition given in \cite{A-H}.
By Lemma 2 of \cite{A-H}, $X_0(N)$ admits a map of degree 3 to $E_1$, which is a contradiction.
\end{proof}

Finally, we consider the hyperelliptic curves $X_0(N)$ of genus $g_0(N)\geq 3$.
Let $\nu$ be a hyperelliptic involution of $X_0(N)$, and
suppose that $X_0(N)$ has infinitely many cubic points.
Then, $X_0^{(3)}(N)$ contains infinitely many rational points.
For each $Q\in X_0(N)$, consider the subset of $X_0^{(3)}(N)$ defined by
$$\Gamma_Q(N):=\left\{(P,\nu(P),Q)\in X_0^{(3)}(N) \,|\,P\in X_0(N)\right\}.$$
Then, $\phi$ maps $\Gamma_Q(N)$ to a single point on $J_0(N)$, because any divisor $P+\nu(P)$ is linearly equivalent to $\infty+\nu(\infty)$. 
Now we set $$U(N):=X_0^{(3)}(N)\,\,\setminus\,\,\bigcup_{Q\in X_0(N)}\Gamma_Q(N).$$

\begin{Lem}\label{hyper1} Suppose that $X_0(N)$ is a hyperelliptic curve of genus $g_0(N)>2$.
Then, the restriction map 
$$\phi|_{U(N)}:U(N)\to J_0(N)$$
is injective.
\end{Lem}
\begin{proof} Suppose that $\phi(P_1,P_2,P_3)=\phi(Q_1,Q_2,Q_3)$ for two different points $(P_1,P_2,P_3),$ $(Q_1,Q_2,Q_3)\in U(N)$.
Then, $[P_1+P_2+P_3-3\infty]=[Q_1+Q_2+Q_3-3\infty]$, and hence the divisor $P_1+P_2+P_3$ is linearly equivalent to the divisor $Q_1+Q_2+Q_3$.
Now, we prove that $P_i\neq Q_j$ for $1\leq i,j\leq 3$.
First, We suppose that it is not true.
Without loss of generality, we may assume that $P_3=Q_3$. 
Then, the divisors $P_1+P_2$ and $Q_1+Q_2$ are linearly equivalent, and hence there is a non-constant function $f$ on $X_0(N)$ whose divisor is equal to $P_1+P_2-Q_1-Q_2$.
Because $f$ has a pole divisor of degree 2, it defines a hyperelliptic involution on $X_0(N)$ that is the same as $\nu$, because the hyperelliptic involution is unique.
Because $f(P_1)=f(P_2)=0$, it holds that $P_2=\nu(P_1)$.
Similarly, we have that $Q_2=\nu(Q_1)$.
Thus, $(P_1,P_2,P_3)$ and $(Q_1,Q_2,Q_3)$ are contained in $\bigcup_{Q\in X_0(N)}\Gamma_Q(N)$, which is impossible.
Thus, the linear equivalence of $P_1+P_2+P_3$ and $Q_1+Q_2+Q_3$ defines a map of degree 3 from $X_0(N)$ to $\P^1$, which is also impossible.
\end{proof}

\begin{Lem}\label{hyper2} Suppose that $X_0(N)$ is a hyperelliptic curve of genus $g_0(N)>2$, and that it admits infinitely many cubic points.
Then, the set $U(N)(\Q)$ of rational points is infinite.
\end{Lem}
\begin{proof}
Suppose that $P$ is a cubic point on $X_0(N)$.
Then, the cubic field $K:=\Q(P)$ has three different embeddings into $\C$.
Considering the Galois closure $L$ of $K$ over $\Q$, one can choose the three embeddings as $\iota,\sigma$ and $\sigma^2$ for some $\sigma\in {\rm Gal}(L/K)$ of order 3 where $\iota$ is the identity. 
Then, $(P,\sigma P,\sigma^2 P)$ is a rational point in $X_0^{(3)}(N)$, and hence the set $X_0^{(3)}(N)(\Q)$ of rational points is infinite.
Now, we will show that $(P,\sigma P,\sigma^2 P)$ does not belong to $\Gamma_Q$ for any $Q\in X_0(N)$.
Suppose that $(P,\sigma P,\sigma^2 P)\in \Gamma_Q(N)$ for some $Q$.
Then, $(P,\sigma P,\sigma^2 P)=(P',\nu(P'),Q)$ for some $P'\in X_0(N)$, and $P'$ is equal to one of $P,\sigma P$ and $\sigma^2 P$.
Without loss of generality, we may assume that $P'=P$ and $\nu(P)=\sigma(P)$.
Then, $P=\nu^2(P)=\sigma(\nu(P))=\sigma^2(P)=Q$, and hence $\sigma$ is of order 2, which is a contradiction.
Thus, such a rational point is contained in $U(N)$.
\end{proof}

By Lemma \ref{hyper1} and Lemma \ref{hyper2}, one can conclude that $W_3(X_0(N))$ contains an elliptic curve with positive rank for hyperelliptic curves $X_0(N)$ of $g_0(N)>2$.
However, there are no elliptic curves with positive rank and conductor $M|N$.
Thus, we obtain the following result:
\begin{Lem}\label{lem5} The curve $X_0(N)$ has at most finitely many cubic points for $N\in\{30,33,35,39,40,41,46,47,48,59,71\}$.
\end{Lem}

Therefore, our main theorem follows from Lemma \ref{lem1}, Theorem \ref{tri}, Lemma \ref{lem2}, Lemma \ref{lem3}, Lemma \ref{lem4}, and Lemma \ref{lem5}.\vspace{0.2cm}

\begin{center}
{\bf Acknowledgment}
\end{center}
This paper was written during the author's sabbatical leave at Brown University.
The author wishes to express his heartfelt thanks to Joseph Silverman for his kind and valuable comments, which helped in proving the results in this paper.
The author also wishes to thank Dan Abramovich and Andreas Schweizer for their comments regarding the proofs of Lemma \ref{lem4} and Lemma \ref{genus2}, respectively.
The author is grateful to the Department of Mathematics of Brown University for its support and hospitality.


\begin{thebibliography}{alpha}

\bibitem{A-H} D. Abramovich and J. Harris, {\em Abelian varieties and curves in $W_d(C)$}, {Compositio Math. {\bf 78} (1991), 227--238.}

\bibitem{B} F. Bars, {\em Bielliptic modular curves}, {J. Number Theory {\bf 76} (1999), 154--165.}

\bibitem{C} J. E. Cremona, {\em Algorithms for Modular Elliptic Curves, Second edition}, {Cambridge Univ. Press, Cambridge, 1997.}


\bibitem{Fa1} G. Faltings, {\em Diophantine approximation on Abelian varieties}, {Ann. of Math. {\bf 133} (1991), 549--576.}

\bibitem{Fa2} G. Faltings, {\em The general case of S. Lang's conjecture}, {Barsotti Symposium in Algebraic Geometry (Abano Terme, 1991), 175--182, Perspect. Math., {\bf 15}, Academic Press, San Diego, CA, 1994.}


\bibitem{D-F} O. Debarre and R. Fahlaoui, {\em Abelian varieties in $W_d^r(C)$ and points of bounded degree on algebraic curves},
{Compositio Math. {\bf 88} (1993), 235--249.}

\bibitem{Fr} G. Frey, {\em Curves with infinitely many points of fixed degree}, {Israel J. Math. {\bf 85} (1994), 79--83.}

\bibitem{H-S} J. Harris and J. H. Silverman, {\em Bielliptic curves and symmetric products}, {Proc. Amer. Math. Soc. {\bf 112} (1991), 347--356.}

\bibitem{Ha-Sh} Y. Hasegawa and M. Shimura, {\em Trigonal modular curves}, {Acta Arith. {\bf 88} (1999), 129--140.}

\bibitem{I-M} N. Ishii and F. Momose, {\em Hyperelliptic modular curves}, {Tsukuba J. Math. {\bf 15} (1991), 413--423.}





\bibitem{JKS} D. Jeon, C. H. Kim and A. Schweizer, {\em On the torsion of elliptic curves over cubic number fields}, {Acta Arith. {\bf 113} (2004), 291--301.}

\bibitem{MS} B. Mazur and P. Swinnerton-Dyer, {\em Arithmetic of Weil curves}, {Invent. Math. {\bf 25} (1974), 1--61.}

\bibitem{Mi} J.~S.~Milne, {\em Modular Functions and Modular Forms,}
{Online notes, available at http://www.jmilne.org/math/CourseNotes/mf.html}

\bibitem{N-Sa} K. V. Nguyen and M.-H. Saito, {\em $d$-gonality of modular curves and bounding torsions,} e-print arXiv: math.AG/9603024, preprint.

\bibitem{O} A. Ogg, {\em Hyperelliptic modular curves}, {Bull. Soc. Math. France {\bf 102} (1974), 449--462.}


\bibitem{O2} A. Ogg, {\em Rational points on certain elliptic modular curves}, {Analytic number theory (Proc. Sympos. Pure Math., Vol XXIV, St. Louis Univ., St. Louis, Mo., 1972), pp. 221--231. Amer. Math. Soc., Providence, R.I., 1973.}

\bibitem{Se} J.-P. Serre, {\em Facteurs locaux des fonctions z\^eta des vari\'et\'es alg\'ebriques (d\'efinitions et conjectures)}, {S\'em. DPP, (1969/70); Oeuvres II, pp. 581--592. Springer, Berlin Heidelberg New York.} 

\bibitem{St} W.~A. Stein, {\em http://www.williamstein.org/Tables}


\end{thebibliography}
\end{document}